\pdfoutput=1
\documentclass{article}
\usepackage[margin=25mm]{geometry}
\usepackage{amsfonts,amsthm,amsmath,amscd,amssymb}
\usepackage{amsfonts}
\usepackage{amssymb}
\usepackage{graphicx}
\usepackage{xcolor}
\usepackage[all]{xy}

\newtheorem{theorem}{Theorem}

\title{Morse flows  with fixed points on the boundary of 3-manifold}
\author{Svitlana Bilun, Alexandr Prishlyak and Andrii Prus}

\begin{document}
\maketitle

\begin{abstract}
The work is devoted to the study of topological properties, structure and classification of Morse flows with fixed points on the boundary of three-dimensional manifolds. We construct a complete topological invariant of a Morse flow, Pr-diagram, which is similar to the Heegaard diagram of a closed 3-manifold.
\end{abstract} 
\hspace{10pt}

%\textbf{\Large{Потоки Морса з нерухомими точками на межі тривимірних многовидів}}

%\\

%\textbf{Пришляк Олександр Олегович.}

%Професор, доктор фіз.-мат. наук.

%Київський національний університет імені Тараса Шевченка, механіко-матема\-тич\-ний факультет.
% м.Київ, проспект академіка Глушкова, 4-Е.

%E-mail: prishlyak@knu.ua

%\textbf{Білун Світлана Володимирівна.}

%Кандидат фіз.-мат. наук, асистент.

%Київський національний університет імені Тараса Шевченка, механіко-матема\-тич\-ний факультет.
% м.Київ, проспект академіка Глушкова, 4-Е.

%E-mail: svbilun@knu.ua

%\textbf{Прус Андрій Анатолійович. }

%Асистент.

%Київський національний університет імені Тараса Шевченка, механіко-матема\-тич\-ний факультет.
% м. Київ, проспект академіка Глушкова, 4-Е

%E-mail:asp00pr@gmail.com

\section*{Introduction}
On every closed manifold, a vector field always generates a flow. In the case of a compact manifold with a boundary, the vector field will generate a flow if and only if it touches the boundary at every point \cite{Loseva2016}.

A vector field $X$ on the manifold $M$ is called \textit{structurally stable} if $X$ has a neighborhood $U$ in the set of all vector fields on the manifold $M$  such that an arbitrary field $Y \in U$ is topologically equivalent to the field $X$.

On closed surfaces, structurally stable vector fields are Morse--Smale fields. For higher-dimensional manifolds, in addition to Morse--Smale vector fields, there are other structurally stable vector fields.
The structural stability of vector fields on closed manifolds was investigated in the works \cite{Peixoto59, Palis68, Palis1970}.

For manifolds with a boundary, an analogue of Morse--Smale fields was described in \cite{Percell_1973, Robinson_1980}.

%Sufficient conditions for structural stability on closed manifolds were obtained (by the Palis-Smale hypothesis \cite{P-S}) Robbin \cite{Rob}, Robinson \cite{Ro}, later Manet \cite{Ma} finished proving the necessary condition for $C ^{1}$ structural stability.

In dimension 2 (closed manifolds), Morse flows (Morse--Smale flows without closed trajectories) have 3 types of fixed points: a source, a saddle and a sink.
If we consider Morse flows with fixed points on the boundary $\partial M$, then
four types of fixed points are possible on surfaces: a source, a sink (Fig. \ref{pic2b}) and two types of saddles (Fig. \ref{pic2nm})

\begin{figure}[ht]
\center{
\includegraphics[height=3.0cm]{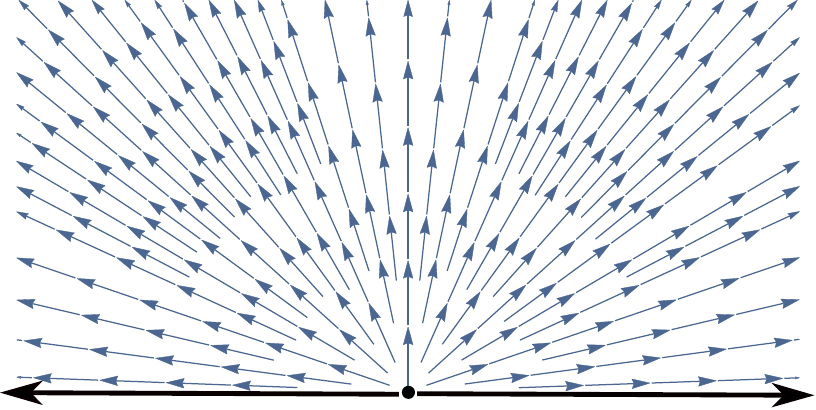} \ \
\includegraphics[height=3.0cm]{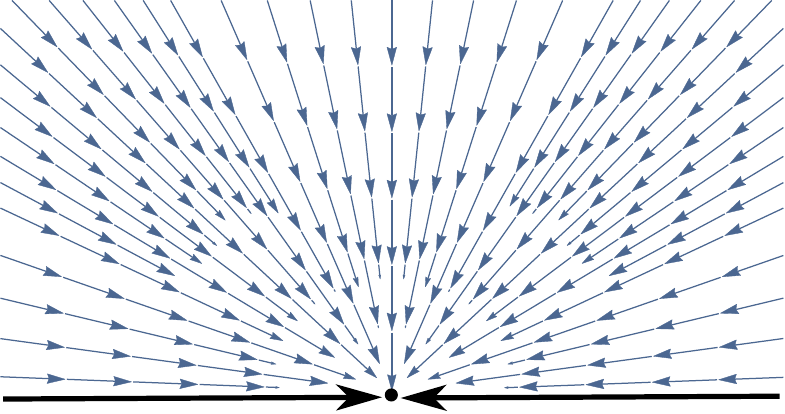}
}
\caption{Source and sink on the surface boundary}
   \label{pic2b}
\end{figure}

% (рис. \ref{pic1}).
%\begin{figure}[h]
%\center{\includegraphics[height=4.5cm]{index2D.eps}}
%\caption{Особливості для розмірності 2}
%\label{pic1}
%\end{figure}

\begin{figure}[ht]
\center{
\includegraphics[height=3.0cm]{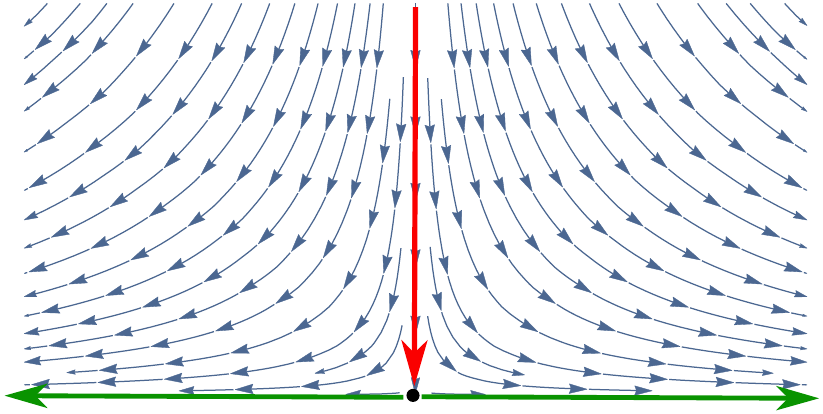} \ \
\includegraphics[height=3.0cm]{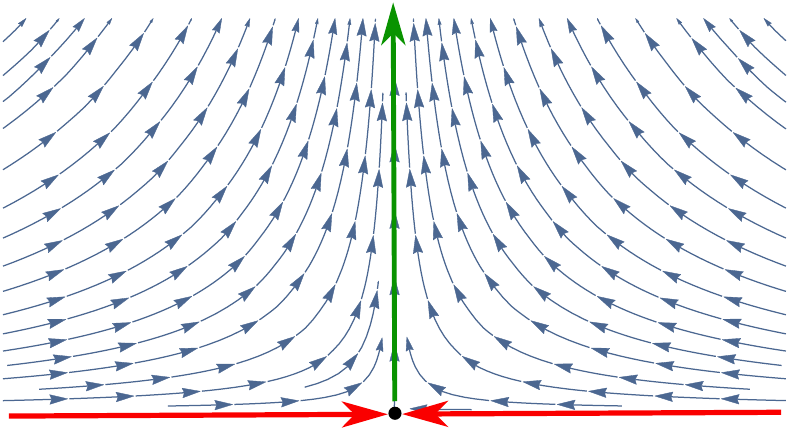}
}
\caption{Saddle point on the surface boundary}
\label{pic2nm}
\end{figure}

Many works have been devoted to the study of topological properties and the construction of a topological classification of Morse--Smale flows on closed surfaces and their generalization to flows with a finite number of special trajectories (fixed points, closed 
 trajectories and separatrixes). Most known of them are \cite{Bolsinov2004, Peixoto73, Fleitas75, Giryk96, kad05, Kruglov2018, Leontovich55, maks11, Poltavec1995, Dibeo-prish, prish03sum}. The main invariant is the separatrix flow diagram. We also note the paper of V.V. Sharko and A.A. Oshemkov \cite{Oshemkov98}, where new invariants are proposed and an overview of other invariants is given, and the work of \cite{Kibalko18}, where the trajectory equivalence of optimal Morse flows is investigated.

On closed three-dimensional manifolds, the topological classification of Morse fields and Morse--Smale fields with some restrictions was obtained by Y.Umanskyi %\cite{Um} 
and A. Prishlyak \cite{prish02ms, prish05ct, prish02top, prish07ct, hp20}.

The paper \cite{prish03tc} gives a topological classification of $m-$fields on 2- and 3-dimensional manifolds with a boundary, which is a generalization of Morse fields and are in general position with the boundary, but such fields do not generate flows.

In recent years, topological properties have been actively investigated and a topological classification for Morse flows on surfaces with a boundary has been constructed. In particular, in the works of M. V. Loseva and O. O. Prishlyak, topological classifications were obtained for flows on a two-dimensional disk with fixed points on the boundary \cite{{Loseva2016}}; for optimal flows on compact surfaces with the boundary \cite{Prishlyak2019, Prus17}, flows with mixed dynamics \cite{PrishLos2020, PPG2021}.
A complete topological classification of Morse flows on surfaces with the boundary using tricolor graphs is given in \cite{Prishlyak2019a}.

Another approach to the classification of Morse flow is connected with Morse functions with the same value of saddle critical point \cite{Smale60, Smale61}. Topological properties of such functions on closed manifolds were investigated in \cite{Kronrod1950, lychak-prish09, prish03pr, prish00, prish02mf, prish02te, Reeb1946} and on manifolds with the boundary in \cite{HladPrish2019, Hladysh-prish17, Hladysh_2019, Hladysh-prish16, Hlad-Prish2020, prish08}.
The main invariants in both cases are graphs embedded in a surface \cite{Prish97}.

\textbf{Purpose.} The purpose of this paper is to construct a complete topological invariant of Morse flows with fixed points on the boundary of an oriented compact three-dimensional manifold, perform a topological classification of these flows using the constructed invariant, and calculate the number of topologically non-equivalent Morse flows on a three-dimensional disk and bodies with handles.

\textbf{The structure of the paper.} The article contains 4 chapters. In \textit{the first chapter}, the process of constructing a complete topological invariant, a Pr-diagram, a Morse flow with fixed points on the boundary of a three-dimensional manifold is described, and examples are given. In \textit{the second chapter} the criterion of topological equivalence of flows due to the isomorphism of their Pr-diagrams is proved. Next, the properties of Pr-diagrams are described, the realization theorem is proved, and the procedure for restoring the flow at the boundary according to the Pr-diagram is explained, as well as the possibility of making global the local continuation of the flow from the boundary to the interior. \textit{The third section} is devoted to the application of the found invariant to calculate the number of topologically non-equivalent Morse fluxes on a three-dimensional disc with no more than six fixed points on the boundary and, as an example of application, the gravitational fl0w of the Sun-Earth system is considered. In \textit{the fourth chapter}, all possible topological structures of optimal Morse flows on a body with two handles and on a body with three handles are described, and the application of the obtained results to water flows in a river with islands is also described.

\section{Pr-diagram of Morse flow with fixed points on the boundary of a three-dimensional manifold}

%\begin{definition}
A flow $X$ on a manifold $M$ with boundary $\partial M$ is called a Morse flow if it satisfies the following conditions:

1) the set of non-wandering points $\Omega(X)$ has a finite number of points, all of them of hyperbolic type;

2) if $u,v \in \Omega(X)$, the unstable manifold $W^u(u)$ is transversal to the stable manifold $W^s(v)$ in $\text{Int} M$;

3) the restriction of $X$ to $\partial M$ is a Morse flow (stable and unstable manifolds have a transversal section).
%\end{definition}

There are 6 types of fixed (singular) points on the boundary of three-dimensional manifolds (Fig. \ref{pic3}). They are determined by their indices.

\begin{figure}[ht]
\center{\includegraphics[height=9.5cm]{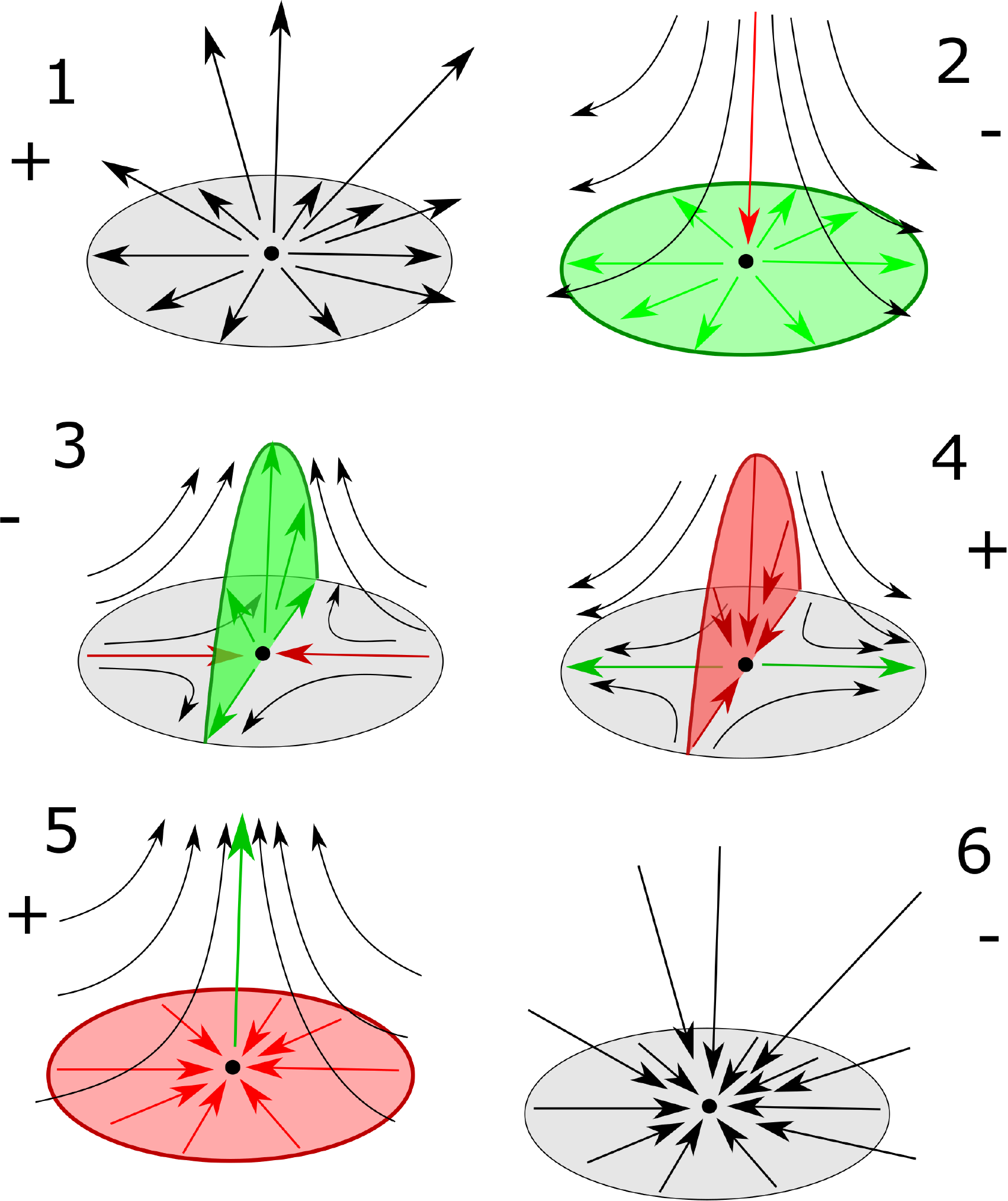}}
\caption{Singularities on the boundary for dimension 3}
\label{pic3}
\end{figure}

The pair $(p,q)$ is called the index of a singular point, where $p+q$ is equal to the dimension of the stable manifold of $X$, and $p$ is the dimension of the stable manifold of the flow narrowed to the boundary. On a three-dimensional manifold $p=0,1$ or 2 and $q=0$ or 1. For example, the source (Fig. \ref{pic3}.1) has index (0,0), and the sink (Fig. \ref {pic3}.6) has index (2,1). Other fixed points in fig. \ref{pic3} have the following indices: 2 -- (0,1), 3 -- (1,0), 4 -- (1,1), 5 -- (2,0). Every Morse flow %on a compact manifold
has a source and a sink.

Let's construct a Pr-diagram of the flow, which has the form of a surface with a boundary and four sets of curves embedded in it.

At the same time, the surface divides the 3-manifold into two parts: one part contains the points of indices (0,0), (0,1) and (1,0), and the second part contains the points of indices that remained. In addition, each trajectory transversally intersects this surface at no more than one point. Such diagrams generalize Heegaard diagrams.

This can be done in two ways: 1) using the axes and co-axes of the distributions on m-handles; 2) using the boundary around a 1-dimensional stable manifold and its intersection with a 2-dimensional stable and unstable manifold.

$M-$handles correspond to fixed points.% \cite{bnr16}. 
So, there are 6 types of m-handles (Fig. \ref{pic7}):

\begin{figure}[ht]
\center{\includegraphics[height=6cm]{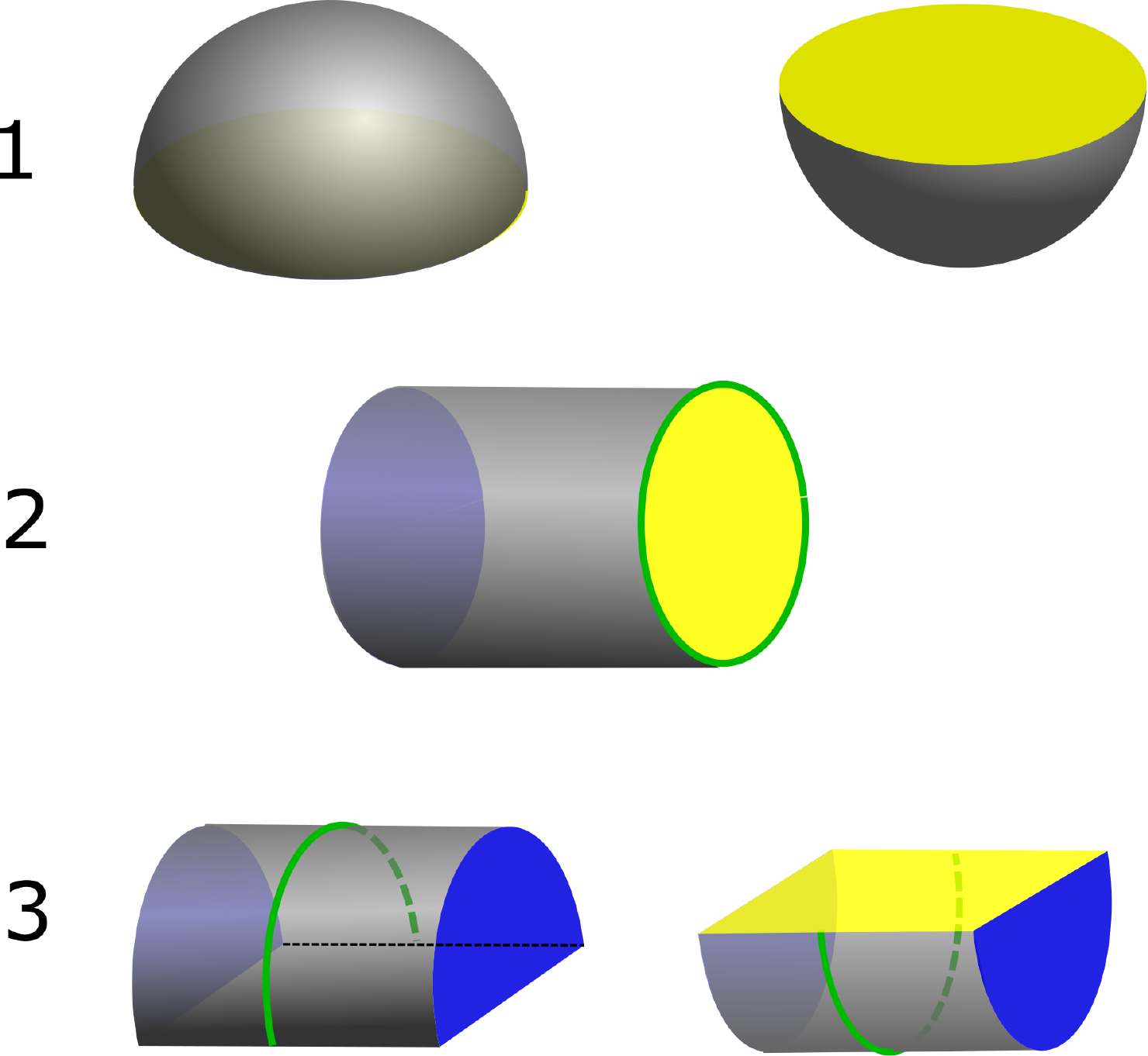}\ \ \ \ \ \  \ \  \  \includegraphics[height=6cm]{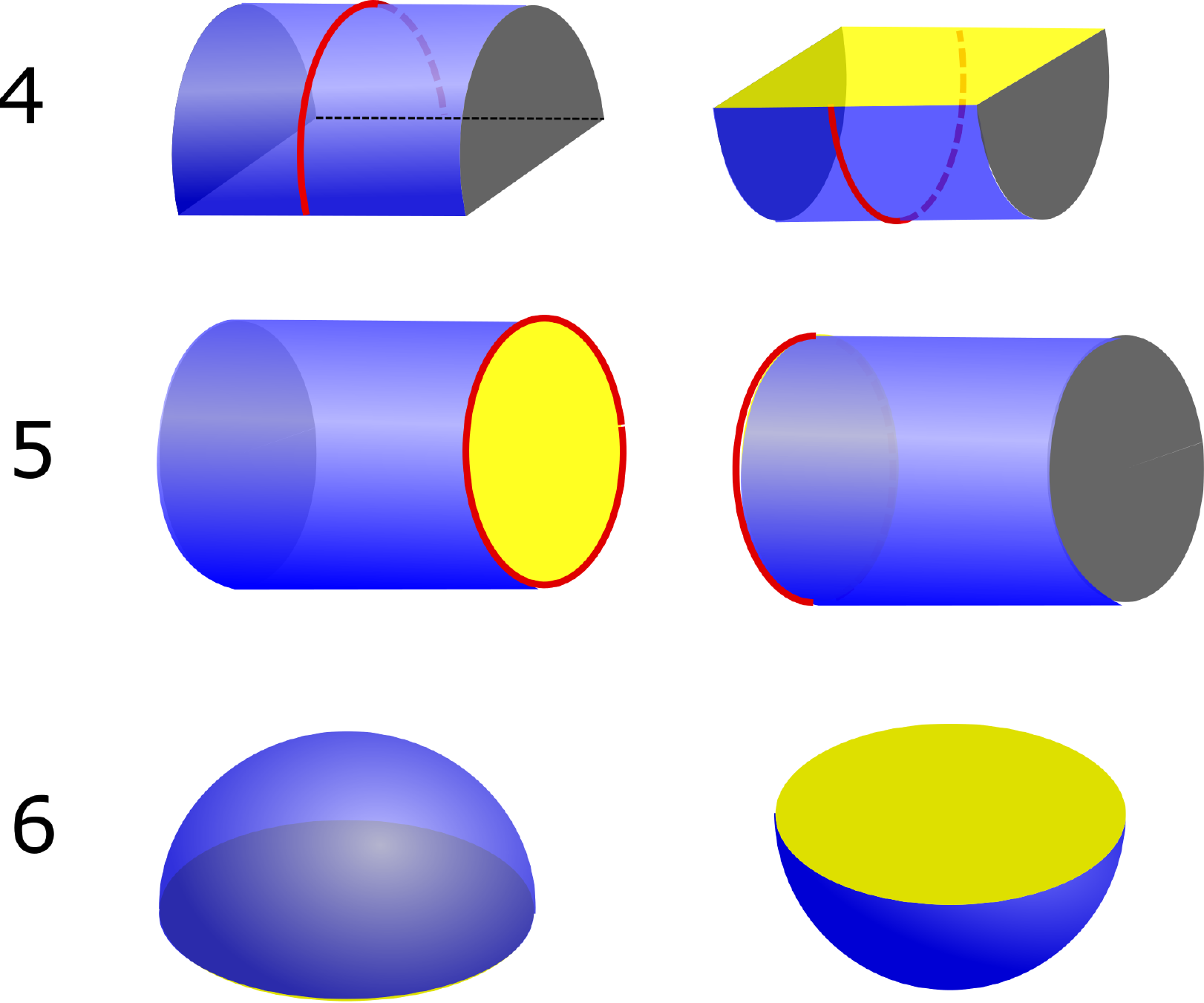}}
\caption{m-handles }
\label{pic7}
\end{figure}

%\begin{figure}[ht]
%\center{\includegraphics[height=6cm]{m-handles4-6.eps}}
%\caption{m-ручки 4-6}
%\label{pic8}
%\end{figure}

%\begin{figure}[ht]
%\center{\includegraphics[height=9.5cm]{mhattach.eps}}
%\caption{m-ручки кріплення}
%\label{pic85}
%\end{figure}

We will describe the construction of the handle decomposition. To get the decomposition into m-handles, we start with handles of type 1 , and then attach others to them as follows: all the blue part of the boundary is attached to the union of the gray part, and its intersections with the yellow part are attached to the boundary of the yellow area. As a result of connecting all m-handles, we get 3-manifolds with a yellow edge. The surface $F$ is the gray part of the union boundary of m-handles of types 1-3. Green pen curves of type 3 form a $u$-system, and green pen curves of type 2 form a $U$-system. The pattern of red curves when connecting handles of the 4-th type forms a $v$-system and red curves of the 5-th type handles forms a $V$-system. If a handle of type 3 is attached to a handle of type 2, then we deform the green part of its section through the blue area and the next gray area. The same can be achieved by compressing the half of the type 3 handle to its middle containing the green arc ($D^2_+\times [-1,1] \to D^2_+\times[-1,0]$, $ (x,t) \to (x,0), t \in [0,1]$). As a result, $u$ becomes part of the cycle $U$. In addition, we apply a similar procedure to the curves $v$ and $V$.

%\begin{definition}
The quintet $(F, u,U, v, V)$ is called the Pr-diagram of the flow $X$, in which $F$ is the gray part of the boundary of the union of handles of types 1, 2, and 3, and the curves $u ,U, v$ and $V$ are as defined above.
%\end{definition}

We will describe another way of constructing Pr-diagrams.

The surface $F$ is the closure of the intersection Int $M$ with a regular neighborhood of the following integral manifolds union:

1) sources, 1-dimensional stable manifolds, and singular points of index (1,0) in $\partial M$;

2) one-dimensional stable manifolds and singular points of index (0,1) in $\text{Int} M$;

 We select the following sets of arcs and circles on the surface $F$:

1) arcs $u$, which are the intersection of unstable manifolds, singular points of index (1,0) and the surface $F$;

2) arcs and circles $U$, which are the intersection of stable manifolds, singular points of index (0,1) and the surface $F$;

3) arcs $v$, which are intersections of stable manifolds, singular points of index (1,0) and the surface $F$;

4) arcs and circles $U$, which are the intersection of stable manifolds, singular points of index (2,0) and the surface $F$.

The set $(F, u, U, v, V)$ , which consists of a surface with an edge, a set of circles and arcs, will be a Pr-diagram of a Morse flow on a three-dimensional manifold with a boundary.

%\begin{example}
\textbf{Ex.}
Let us give an example of a Pr-diagram of a Morse flow.
Gradient flow of the height function on the complete torus (Fig. \ref{pic9}).

\begin{figure}[ht]
\center{\includegraphics[height=5.5cm]{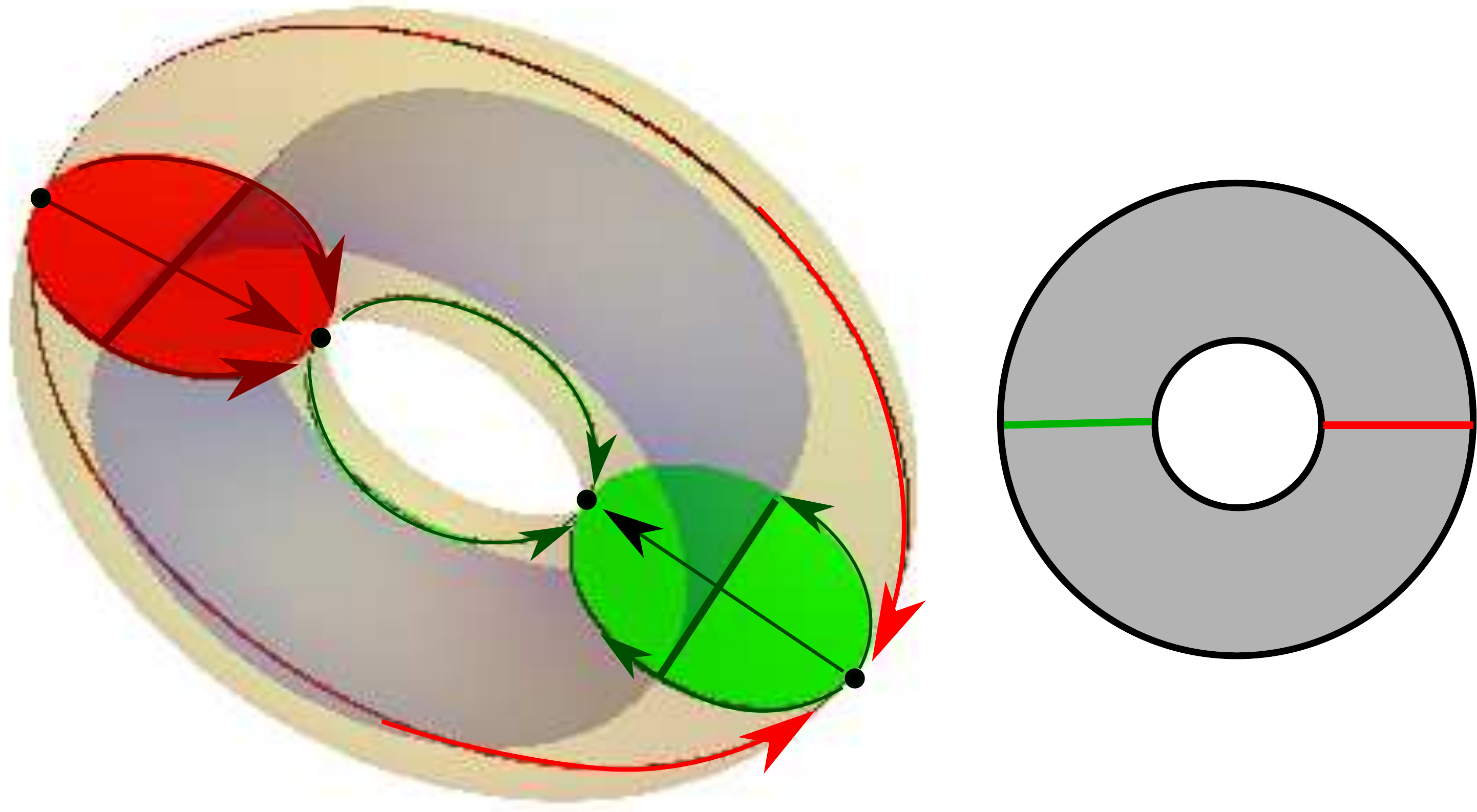}}
\caption{Optimal Morse flow on a full torus. }
\label{pic9}
\end{figure}

The height function on a complete torus has 4 critical points when limiting it to the boundary. The corresponding fixed points  have the first, third, fourth and sixth types.

Such flows occur during inflation of a flat tire as air flows inside the wheel.

In fig. \ref{pic111} shows the process of constructing Pr-diagrams for three flows on a three-dimensional disk with 6 fixed points.

\begin{figure}[ht]
\center{\includegraphics[height=8.5cm]{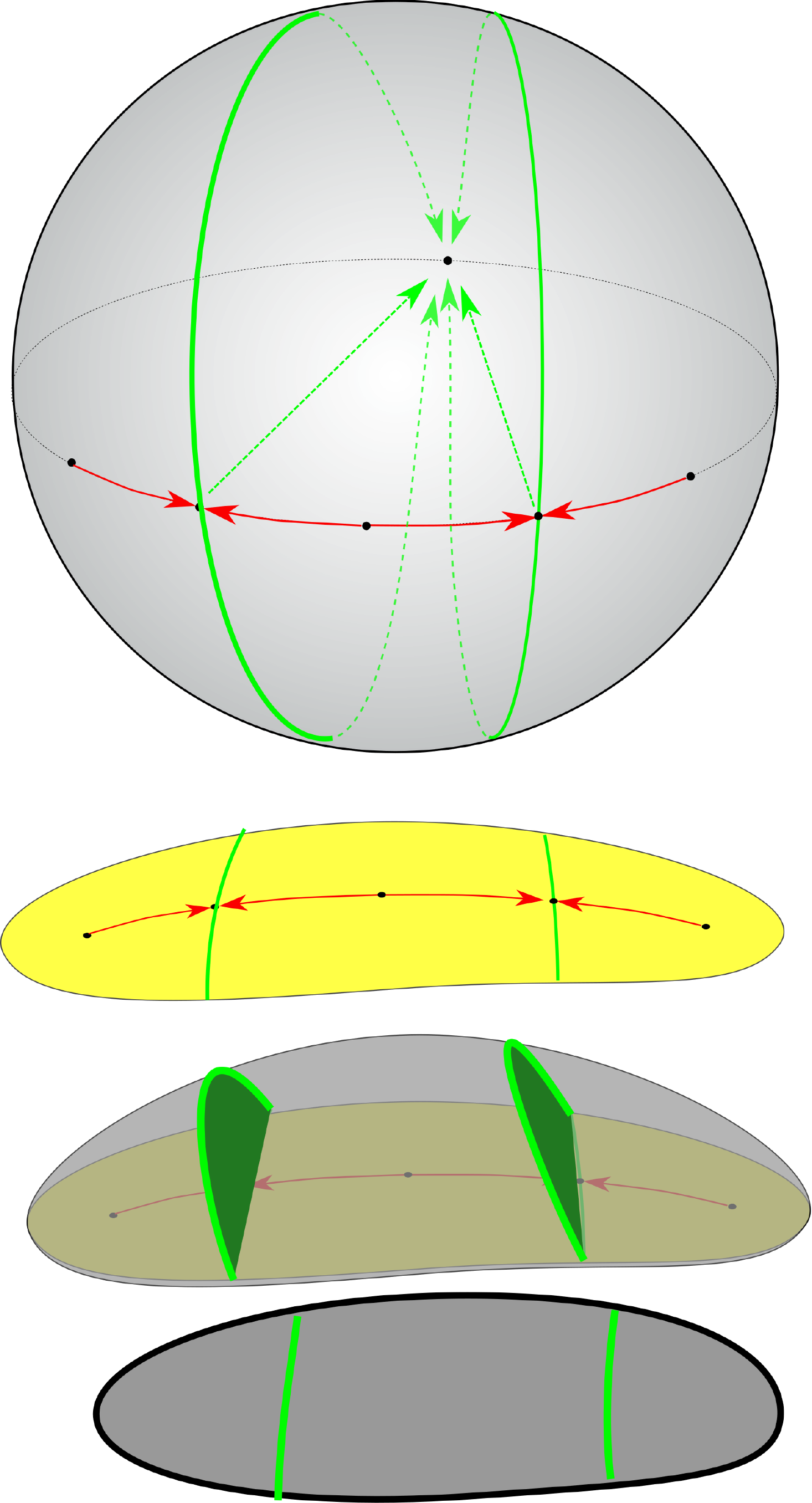} \ \ \ \ \includegraphics[height=8.5cm]{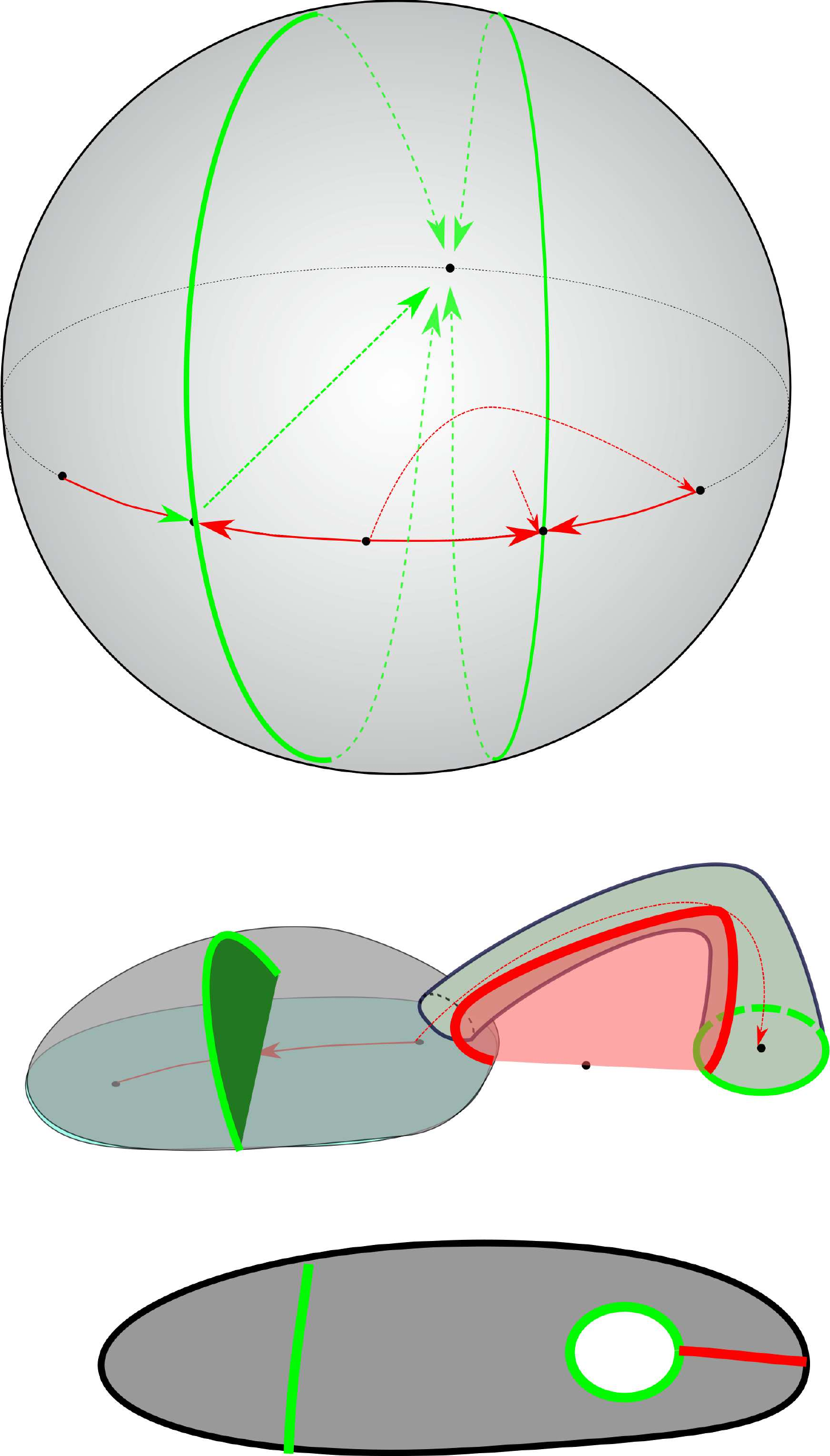} \ \ \ \ \ \includegraphics[height=8.5cm]{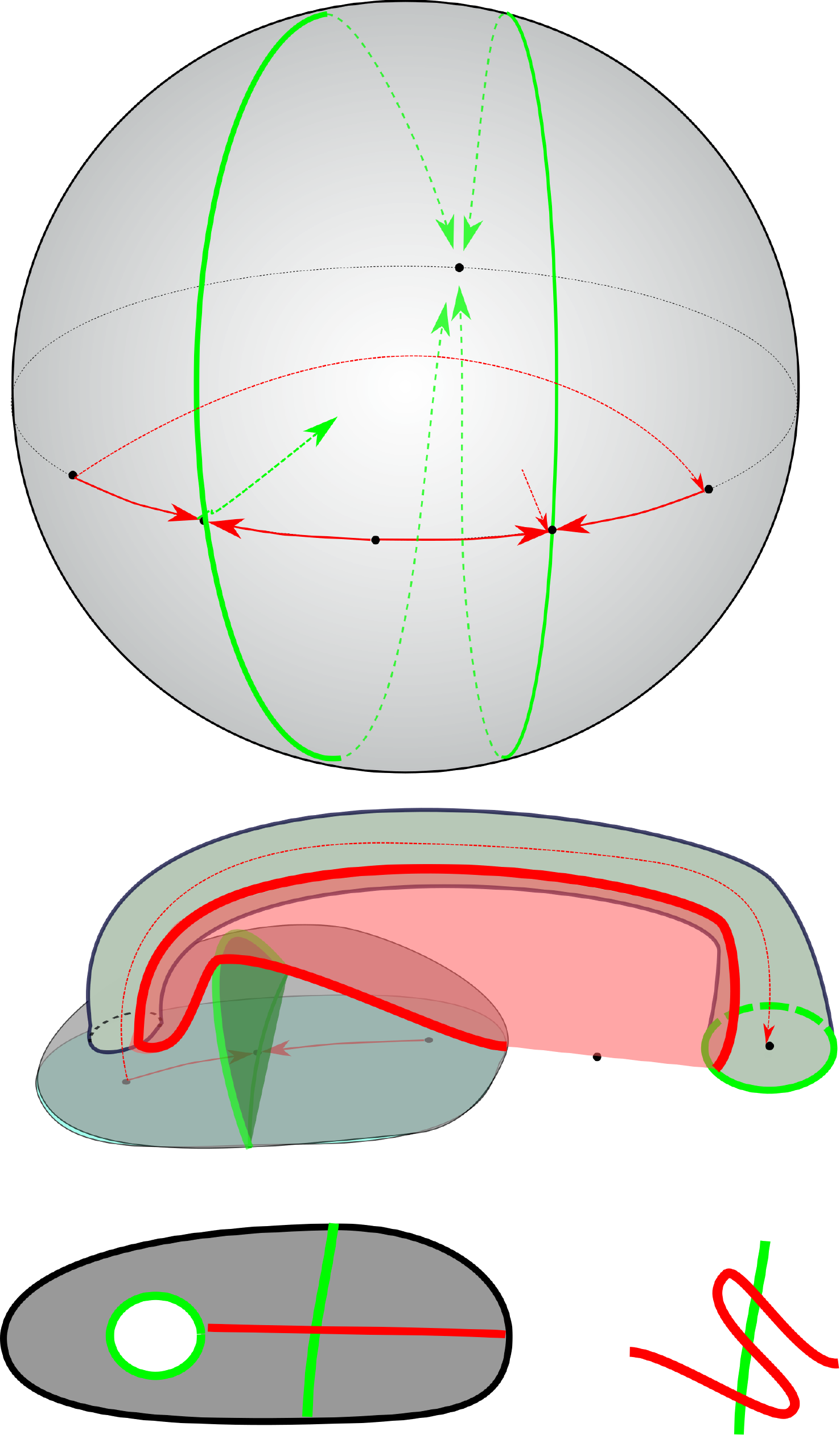}  }
\caption{Flows on $D^3$}
\label{pic111}
\end{figure}

Depending on the intersection of stable and unstable manifolds of dimension 2 inside the 3-disc, the final Pr-diagram could have 3 or more points of intersection between the red and green curves.

%\end{example}

\section{Classification of Morse flows according to Pr-diagrams}

To classify Morse flows according to the Pr-diagram, we will prove the criterion of topological equivalence of flows, describe the topological properties of the diagrams and determine which of the diagrams can be implemented by Morse flows. In addition, if the Pr-diagram defines a flow on a three-dimensional manifold, then it also defines a flow that is its boundary restriction. Let's set the topological type of this stream. Also consider the inverse problem, can the local continuation of the flow from the boundary be made global?

\subsection{The criterion of topological equivalence of flows}
Two Pr-diagrams of a Morse flow are isomorphic if there exists a surface homeomorphism that maps sets of arcs and circles into sets of arcs and circles of the same type.

%Prove that diagrams constructed in different ways are isomorphic.

\begin{theorem}
Two Morse-Smale flows on a three-dimensional manifold with the boundary are topologically trajectory equivalent if and only if their Pr-diagrams are isomorphic.
\end{theorem}
\begin{proof}
\textit{Necessity.} Let $\varphi: M\to M $ be the topological equivalence of the flow generated by the field $X$ to the flow generated by the field $X^{\prime}$. Let us construct the homeomorphism $h:F\to F^{\prime}$ between the surfaces of their Pr-diagrams. If $\gamma_x$ is a trajectory passing through the point $x\in F$, then the desired homeomorphism is given by the formula $h(x)= f(\gamma_x)\cap F^{\prime}$.

\textit{Sufficiency. } The proof is similar to the proofs for closed manifolds in dimensions 2 and 3, so we only give its scheme. If the isomorphism $h:F\to F^{\prime}$ of the Pr-diagrams is given, then it establishes a bijection between the trajectories of the vector field: the trajectories crossing the surface $F$ are matched with the trajectories passing through the images of intersections, and the trajectories , belonging to 1-dimensional stable or unstable manifolds, correspond to the trajectory, which is determined by mapping the surroundings of the curves $u,U,v,V$. The surface $F$ splits the manifold $M$ into two parts: the first part $M_1$ contains those semi-trajectories that enter the points from $F$, and the second part $M_2$ contains the semi-trajectories that leave $F$. Let us construct a topological equivalence mapping $M_1$ to $M_1^{\prime}$. We choose such Riemannian metrics on these manifolds that the corresponding trajectories and semi-trajectories belonging to 1-dimensional and 2-dimensional stable and unstable manifolds have the same length. Let's construct a homeomorphism of the corresponding trajectories and semi-trajectories that preserves the ratio of the lengths of the arcs of the curves. Together, these homeomorphisms will define the desired homeomorphism $M_1$ on $M_1^{\prime}$. Similarly, we construct the homeomorphism $M_2$ on $M_2^{\prime}$. Since these two homeomorphisms coincide on a common boundary, they define the desired homeomorphism $M$ on $M^{\prime}$.
\end{proof}
\subsection{Topological properties of Pr-diagrams of Morse flows}
\begin{theorem}
Pr-diagrams of Morse flows have the following properties:

1) $U_i, V_i \subset \partial M$, $\text{Int } u_i, \text{Int }  v_i \subset \text{Int } M$, $\partial u_i, \partial  v_i \subset \partial M$;

2) $\partial U_j \subset \cup_i \partial u_i,  \partial V_j \subset \cup_i \partial v_i$;

3) $ U_i \cap U_j = \emptyset $, if $i \ne j$, \ \
$ u_i \cap u_j = \emptyset $, if $i \ne j$, \ \
$ V_i \cap V_j = \emptyset $ , if $i \ne j$, \ \
$ v_i \cap v_j = \emptyset $ , if $i \ne j$, \ \
$\text{Int} u_i \cap \text{Int} U_j = \emptyset $, \ \
$ \text{Int } v_i \cap \text{Int } V_j = \emptyset $, \ \
$\partial u_i   \cap \partial v_j = \emptyset $.

4) $U_k$ is a closed curve or belongs to a cycle consisting of $U_i$ and $u_j$, and such that $u_j$ always turns to the left at the ends of the arcs; a similar property is also true for $V_k$.

5) if we cut the surface $F$ along $u_i$ and make a spherical rearrangement according to $U$-cycles, we will get a union of 2-disks. The same is done for $v_i$ and $V_j$.
\end{theorem}

%\begin{remark}
 \textbf{Remarks.}
 The last 2-disks ($u$-area) correspond to sources (type 1), $U$-cycles correspond to points of index (0,1) (type 2), $u$-curves to points of index (1,0) (type 3), $v$-curves -- points of index (1,1) (type 4), $V$-cycles -- points of index (2,0) (type 5), $v$-areas -- drains (type 6).
%\end{remark}

\begin{proof}
1) This follows from the fact that the corresponding trajectories of stable and unstable manifolds belong to the interior or boundary of the manifold $M$.

2) Boundaries of unstable manifolds of index (0,1) consist of curves $U_i$ and intersections of unstable manifolds with $\partial M$. Since the boundaries of these intersections coincide with $\partial u_i$, we have the first inclusion. The second inclusion is proved similarly.

3) Conversely, if there were a point belonging to these intersections, the trajectory passing through it would belong to two different stable (or two unstable) manifolds of dimension 2, which is impossible.

4) These cycles correspond to the boundaries of unstable manifolds of index (0,1) or stable manifolds of index (2,0).

5) These disks correspond to unstable manifolds of points of index (0,0).
\end{proof}

\subsection{Implementation Theorem}

\begin{theorem} If a surface $F$ with 4 sets of curves has properties 1-5, then it is a Pr-diagram of a Morse flow.
\end{theorem}
\begin{proof} Since by property 3) $ u_i \cap u_j = \emptyset $, if $i \ne j$, then we can choose sufficiently small non-intersecting regular neighborhoods of them. Let's glue the handles of the third type on the gray areas to these circles (at the same time, the green curves are displayed on the $u_i$ curve). If these curves were included in $U$-cycles, then we will carry out the deformation of these cycles, which is inverse to the deformation of $U$-cycles when gluing the handle of type 3 to the handle of type 2. Next, for these cycles, we will select the part of the surface that remained on the union, and the blue areas of the glued pens are regular circles. We will glue the gray areas of the 2nd type handles to these circles. The parts of the surface that remained, combined with the blue areas, will be sealed with the gray area of the handles of the first type (this can be done according to property 5)). We will carry out similar constructions for
$v$- and $V$-curves by gluing handles of types 4, 5, and 6 to the surface. Inside each handle, we set the standard vector field $\{ \pm x, \pm y,\pm z \}$ as in fig. \ref{pic3}. After smoothing these fields in the places of gluing, we will get the desired vector field.
\end{proof}

%Наведено діаграми потоків Mорса з 4 і 6 особливими точками на межі $D^3$, оптимальні потоки Морса на корпусах ручок і добутку закритої поверхні і [0,1], а також потоки Морса на доповненнях вузлів.

\subsection{Restoration of flow at the boundary according to the Pr-diagram.}

We will show how the topological type of flow restriction on the boundary of a three-dimensional manifold can be determined by the Pr-diagram.
We present the boundary of the 3-manifold in the form of the union $\partial M = F_u \cup F_v$, where $\partial F_u = \partial F_v = F_u \cap F_v$. $F_u$ contains fixed points of types 1--3 and the flow on it is determined by $F$ and curves $u_i$, $U_j$. $F_v$ contains fixed points of types 4--6 and the flow on it is determined by $F$ and curves $v_i$, $V_j$. At the common border, the flow is directed from $F_u$ to $F_v$. To construct $F_u$, consider $U$-cycles on $F$ and their regular neighborhoods. Let $U^1_i$ be those components of the boundary of these neighborhoods that do not intersect with $\partial M$. Let's carry out spherical reconstructions along them - cut $F$ along them and glue the resulting closed curves with 2-discs. We denote the resulting surface by $F_u$. The center (arbitrary internal point) of each curve $u_i$ will be the saddle point of the flow, and the curve itself will be an unstable manifold for it (the union of two separatrixes). In each region, into which $u_i$ curves divide $F_u$, we will choose a source and draw one separatrix to each saddle that lies on the border of this region. All other trajectories will go from the sources to $\partial F_u$. We will carry out similar constructions with $F$ and curves $v_i$, $V_j$. We obtain the surface $F_v$ and the flow with saddle points and drains. As a result, we get a flow on $\partial M$ as a union of constructed flows. After smoothing on the boundary $\partial F_u$, we will get the desired flow (Smoothing can be omitted if a flow perpendicular to its boundary is built on each surface).

\subsection{On the continuation of the flow from the boundary neighborhood to the interior of the 3-manifold}

The topological type of continuation of the flow from the boundary of the 3-manifold to the regular neighborhood of the boundary depends on the type of singular points (as points of the 3-manifold), which are defined as whether there is a trajectory belonging to the interior and entering the singular point. If it is, then the second number in the index is 1, and if it is not (which is equivalent to the fact that there is a trajectory that starts from a special point), then it is 0. In addition, the sum of the Poincaré indices must be equal to 0 (the Euler characteristic of the doubling of the manifold).

\begin{figure}[ht]
\center{\includegraphics[height=5.5cm]{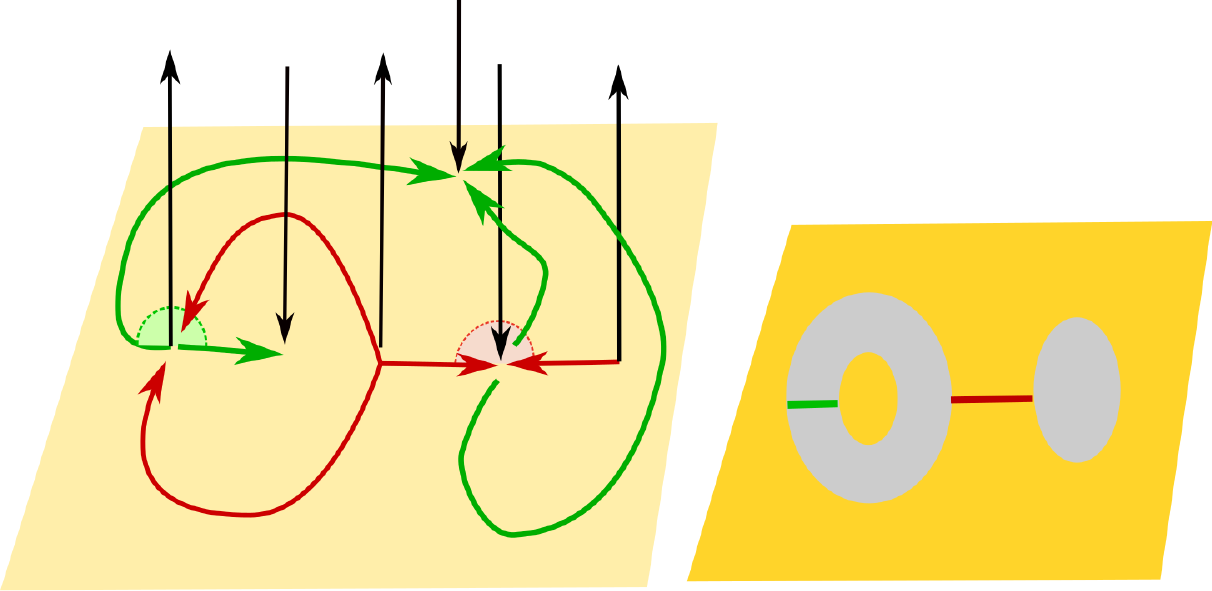}}
\caption{Impossible flow on $D^3$}
\label{pic96}
\end{figure}

The existence of a global extension to the entire 3-manifold is equivalent to the existence of a flow diagram for which the procedure described in the previous section will give the initial flow on the surface. The construction of the $F$ surface and the $u_i$, $U_j$ curves is similar to that for the given flow. At the same time, there is arbitrariness in the choice of sources from which the trajectories that go to points of type (0,1) originate. After carrying out a similar procedure with the curves $v_i$, $V_j$, we will get another surface with the same boundary. Then the task of constructing a Pr-diagram consists in finding a homeomorphism of one surface onto another, which coincides with the identical reflection on the boundary. This is equivalent to the construction on the first surface with the curves $u_i$,$U_j$ of the curves $v_i$ with their specified points on the boundary (the curves $V_j$ already lie on the boundary). At the same time, the constructed Pr-diagram should specify a 3-disc. This means that after gluing to the boundary components of the 2-discs, the $U-$ and $V-$ loops will form the meridian system of the Hehor diagram of the three-dimensional sphere.

In fig. \ref{pic96} an example is given when the identical homeomorphism of boundaries (three circles) cannot be extended to the homeomorphism of surfaces (2-disc combined with a ring). Thus, this flow around the sphere does not continue to the 3-disc.

\section{Morse flows on a three-dimensional disk}

In order to find all flows with a given set of fixed points on the boundary of the 3-disc, we will find all possible extensions from the boundary. At the same time, we will assume that the flows do not have internal curvilinear diagonals with green and red sides. Since the Euler characteristic of a closed surface is even, the number of fixed points on the boundary is also even. With two fixed points (source and drain), there is a single flow. Its Pr-diagram is a two-dimensional disk without red and green curves.

\subsection{Flows with four fixed points on the boundary}

It follows from the definition of a Morse flow that it has a source and a sink (points of type 1 and 6). A flow restriction on a sphere can have two sources, a saddle and a sink, or a source, a saddle, and two sinks. These two flows are obtained from each other with the accuracy of topological equivalence by changing the directions of movement along all trajectories. Therefore, in the future, we will consider only those flows in which the boundary restriction has at least sources and sinks. Then two cases are possible: 1) an internal trajectory emerges from the saddle, and therefore from one of the sources, on the sphere, 2) internal trajectories enter these points. Pr-diagrams of these flows are shown in fig. \ref{pic98}.

\begin{figure}[ht]
\center{\includegraphics[height=2.2cm]{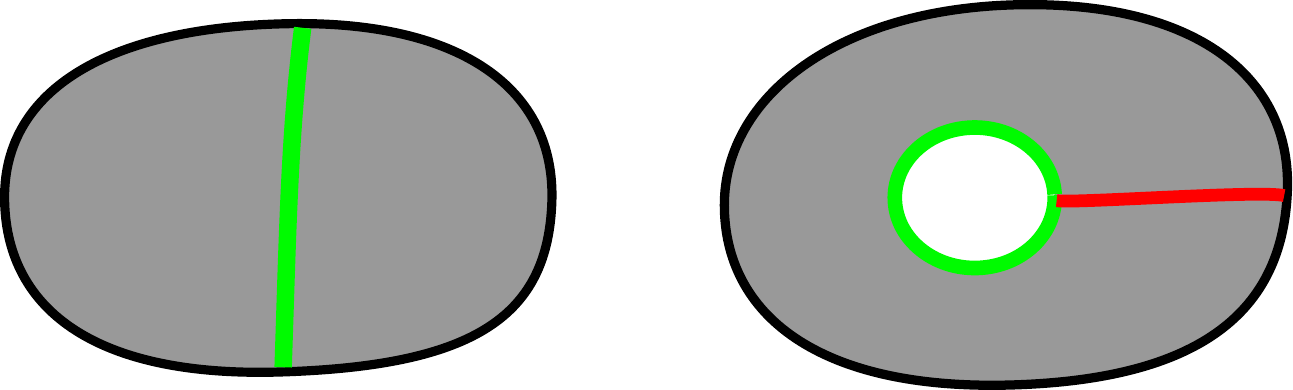}}
\caption{Two flows with four fixed points on $D^3$}
\label{pic98}
\end{figure}

Therefore, there are four topologically inequivalent flows with four fixed points.

\subsection{Flows with six fixed points on the boundary}

In the case of six fixed points, the following flows on the sphere are possible: 1) three sources, two saddles and a sink, 2) two sources, two sinks and two saddles, one of the stable one-dimensional manifolds forms a loop, 3) two sources, two sinks and two saddles, stable one-dimensional manifolds do not form loops, 4) one source, three sinks and two saddles. The number of continuations of flows in the fourth version is the same as in the first (because they are obtained from each other by changing the directions of movements on the trajectories). Three flows of the first type have already been described earlier (Fig. 6). If a trajectory enters the middle source on the sphere from the right source, then the possible flow is as shown in fig. \ref{pic10}.

\begin{figure}[ht]
\center{\includegraphics[height=4.5cm]{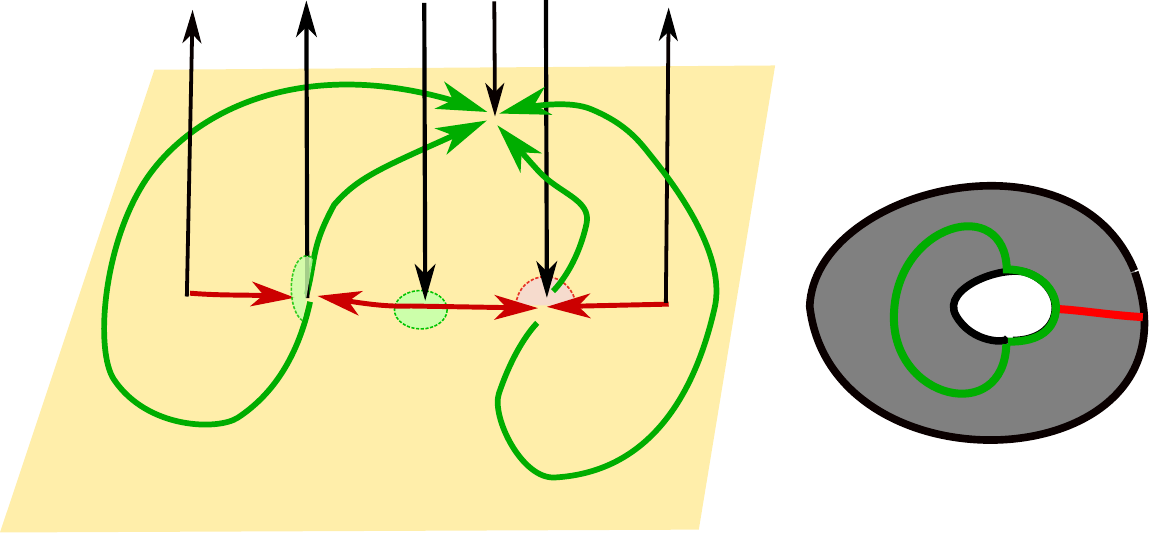}}
\caption{The flow on $D^3$ of the first type and its Pr-diagram}
\label{pic10}
\end{figure}

If the trajectory entering the middle fixed point would come from the left source (and all directions are as in Fig. \ref{pic10}), then we would get a flow that cannot be extended into the middle of the 3-disc.

If only one of the sources on the sphere is a source on the three-dimensional disk, then two options are possible: this source is in the center (Fig. \ref{pic11} on the left) or on the side (Fig. \ref{pic11} on the right).

\begin{figure}[ht]
\center{\includegraphics[height=2.5cm]{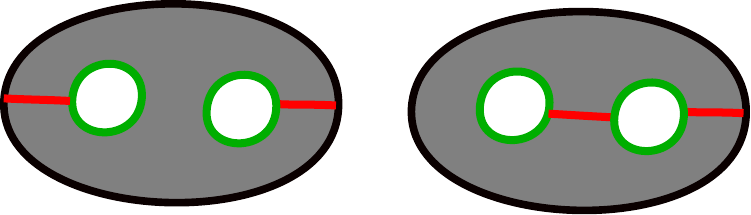}}
\caption{Two Pr-flow diagrams of the first type}
\label{pic11}
\end{figure}

Therefore, for the first type of flows on the sphere (and therefore also for the fourth), there are 6 non-equivalent extensions to the flow on the 3-disc.

Next, we will consider the second type of flows on the sphere. Two examples of the continuation of such flows are shown in fig. \ref{pic13b}.

\begin{figure}[ht]
\center{\includegraphics[height=4.7cm]{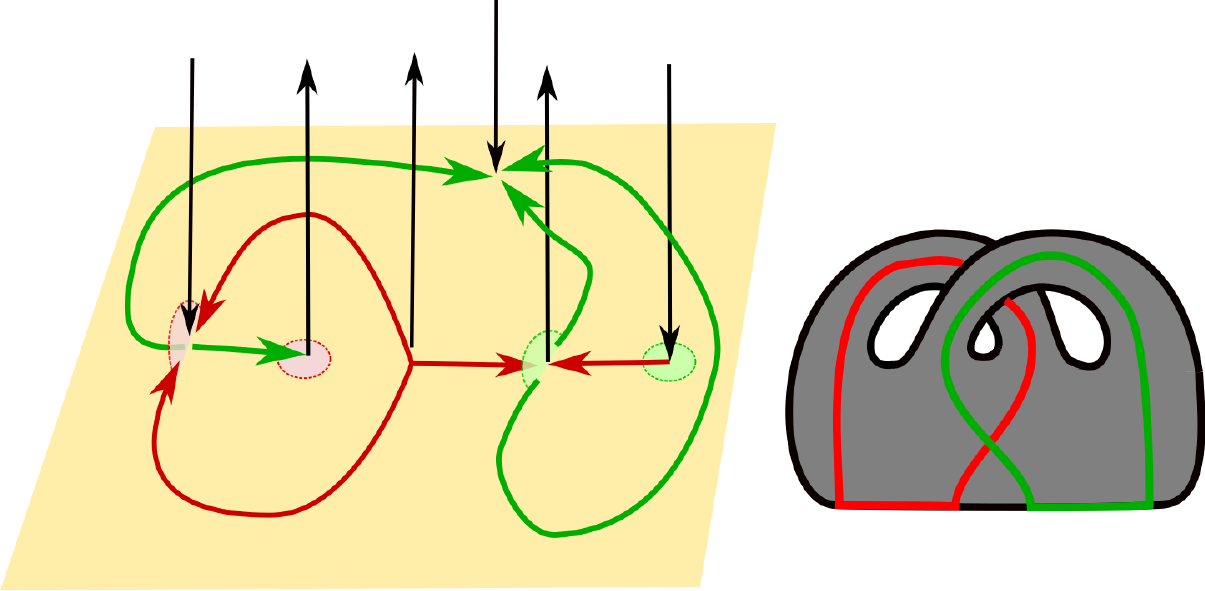}}

%\caption{Потік на  $D^3$}
%\label{pic13a}
%\end{figure}

%\begin{figure}[ht]
\center{\includegraphics[height=4.7cm]{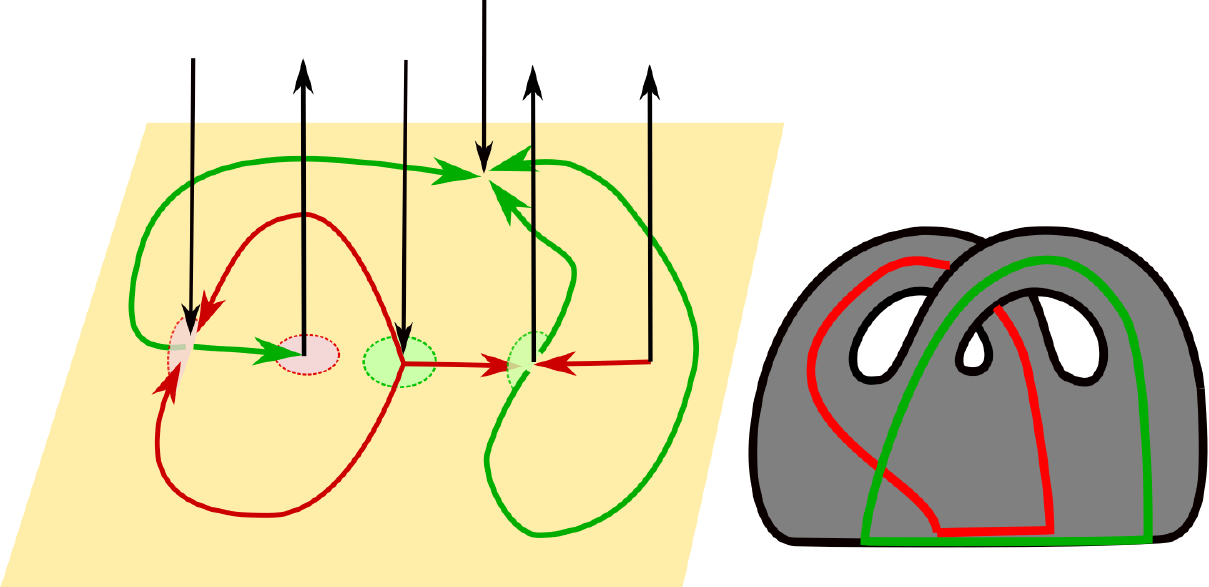}}
\caption{Two flows on $D^3$ of the second type and their Pr-diagrams}
\label{pic13b}
\end{figure}

Taking into account the restriction on the sum of indexes, there are only 14 possibilities to locally continue this flow, one of which cannot be implemented globally (see Fig. 6). Each other can be implemented uniquely, so we have 13 flows in this case.

\begin{figure}[ht]
\center{\includegraphics[height=2.5cm]{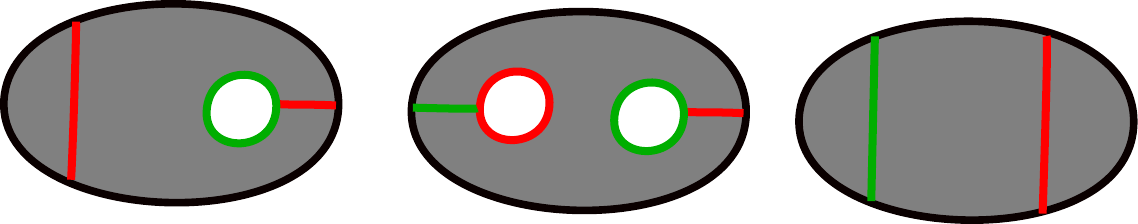}}
\caption{Three diagrams of the second type}
\label{pic13a}
\end{figure}

Three more of the Pr-diagrams of these flows are shown in fig. \ref{pic13a}. Two more Pr-diagrams can be obtained from the Pr-diagrams in fig. \ref{pic13b}, if the red coloring / no border coloring is changed to the opposite. In addition, three more Pr-diagrams can be obtained from the first two Pr-diagrams of fig. \ref{pic13a}, if the following operations are allowed: 1) simultaneously change the red and green colors of all arcs and chords, 2) change the red color of the border (painted - not painted).

\begin{figure}[ht]
\center{\includegraphics[height=5cm]{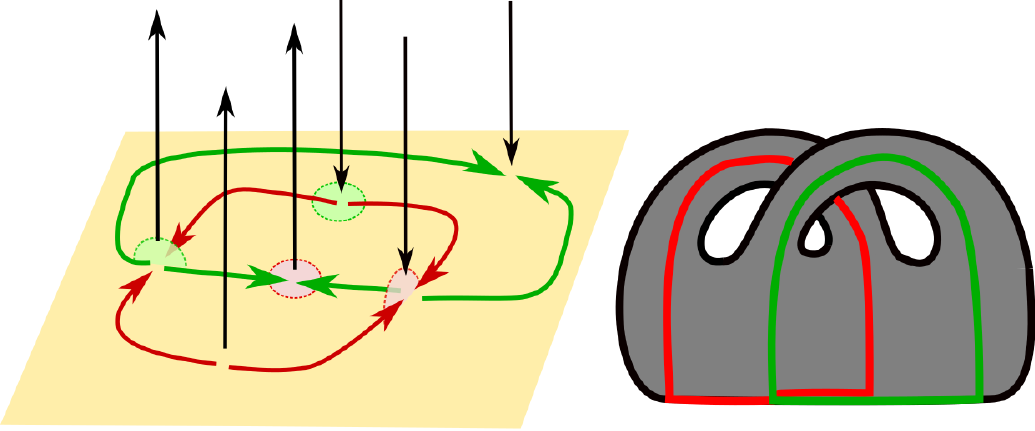}}
\caption{A flow on $D^3$ of the third type and its diagram}
\label{pic13c}
\end{figure}

Next, we will consider the third type of flows on the sphere. This flow can be represented on the unit sphere so that it is symmetric with respect to the coordinate planes. Given the symmetry, such flows will be determined by the number of global sources and sinks. So, 4 situations are possible: 1) one source and one drain (Fig. \ref{pic13c}), 2) two sources and one drain (Fig. \ref{pic13d} on the left), 3) one source and two drains (Fig. \ref{pic13d} in the center), 4) two sources and two sinks (Fig. \ref{pic13d} on the right).

\begin{figure}[ht]
\center{\includegraphics[height=3cm]{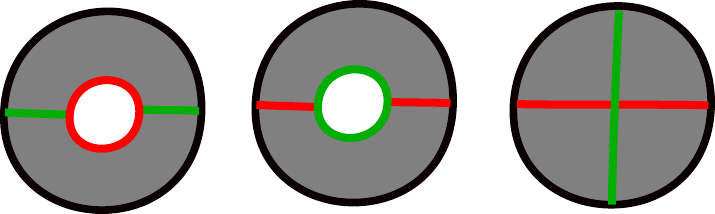}}
\caption{Three diagrams of the third type}
\label{pic13d}
\end{figure}

Summing up all options, we have 29 flows with 6 fixed points on the boundary of $D^3$.

%\begin{center}
%\textbf{From Heegaard diagram to Morse flow diagram}
%\end{center}
%\begin{figure}[ht]
%\center{\includegraphics[height=9cm]{Hee2md.eps}}
%\caption{From Heegaard diagram to Morse flow diagram}
%\label{pic13k}
%\end{figure}
\subsection{Gravitational flow of the Sun-Earth system}

In the Sun-Earth system, we consider points moving around the Sun (more precisely, the center O of the masses of the Earth and the Sun) with the same angular velocity as the Earth. Three forces act on the points -- gravity towards the Sun, gravity towards the Earth and centrifugal force. The first two forces are inversely proportional to the modulus of the distance to the corresponding bodies, and the third is proportional to the distance to O. Then the centers of the Earth and the Sun are drains. There are 3 more fixed points on the Sun-Earth line (Lagrange points $L_1, L_2, L_3$), which are saddles in the rotation plane S. In addition, in this plane there are two more fixed points ($L_4,L_5$) -- sources.

\begin{figure}[ht]
\center{\includegraphics[height=7.5cm]{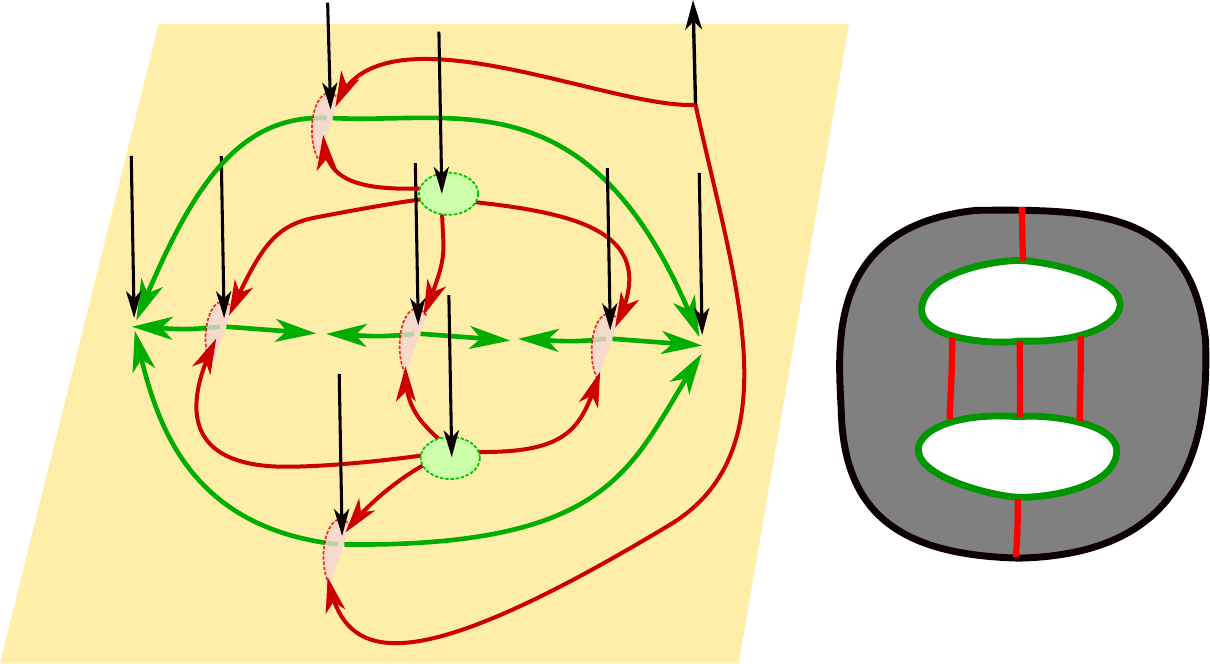}}
\caption{Flow of the Sun--Earth system on $D^3$}
\label{pic14}
\end{figure}

The rotation of the $S$ plane uniquely generates a rotation of the three-dimensional space around an axis perpendicular to $S$ passing through O. Since the system is symmetric with respect to this plane, it is described by a flow in one of the two subspaces, which we will call the upper one. Consider a hemisphere centered at O, with a sufficiently large radius that all fixed points lie inside it in union with a flat disk of the same radius. We will smooth this union at the intersection, and at the points of the resulting surface we will project the field onto the tangent plane and smooth it. This process corresponds to the fact that at sufficiently large distances we are not interested in the distance to the point, but only in its placement on the telescope screen, and smoothing at the intersection points is the movement of the observer to the inner point of the upper subspace.

At the same time, there will be a spring at the pole in the upper hemisphere, and two more drains and two saddles will be added at the equator. This flow and its diagram are shown in fig. \ref{pic14}. On the Pr-diagram, the green curves correspond to the points $L_4$ and $L_5$, the three red curves between them correspond to the points $L_1,L_2,L_3$, and the curvilinear quadrilaterals bounded by them correspond to the Sun and the Earth.

\section{Optimal Morse flows on a body with handles.}
%\begin{definition}
A Morse flow with fixed points on the boundary of a body with handles $M$ will be called optimal if it has the minimum number of singularities and saddle connections among all such flows on $M$.
%\end{definition}

It follows from the definition that the optimal flow has no fewer fixed points than the optimal flow on the boundary. In addition, the Pr diagram of the optimal flow has the least number of intersection points between the red and green meridians.

\begin{theorem} On bodies with $g$ handles, the Morse flow is optimal if and only if it has one point of type 1 and one of type 6 each, and $g$ curves $u$ and $v$ (fixed points 3 and 4 types) and without intersection points between them in the flow diagram.
\end{theorem}

\begin{proof}
\textit{Sufficiency.} Since every Morse stream has a source and a drain, it has points of type 1 and 6. Since only the gluing of handles of the third type increases the genus of the surface (body), then the body with $g$ handles must have at least $g$ points of the third type (curves $u$). If we reverse the direction of movement along the flow, then by the same reasoning we will obtain that there are no fewer than $g$ points of the fourth type.
\textit{Necessity.} Let us show that there exists a flow that satisfies the conditions of the theorem. For this, we will show its diagram. It is a 2-disc with $g$ holes, each hole connected to the boundary of the disk by a pair of parallel curves, one red and the other green. For $g=2$ see Fig. \ref{pic14b} first chart on the left.
\end{proof}

Since optimal flows on connected closed oriented surfaces can be determined using chord diagrams, we will show how they can be used to find flows on a body with handles.

If we cut the $F$ surface of the Pr-diagram along the red curves, we get a 2-disc with green chords. We connect the middle of the red sides with a red chord, if it belongs to the same red curve. The resulting colored chord diagram is a complete topological invariant of the flow.

Therefore, a colored (with red and green chords) chord diagram will be an optimal Morse flow chord diagram if it is 1-cycle, the number of red and green chords is equal, and the green chords have no intersections.

Note that if we omit coloring on the colored chord diagram (all chords are colored in one color), then we get a chord diagram of the limit flow. Therefore, in order to find the diagrams of all optimal flows on a body of the genus $g$, it is necessary to select $g$ non-intersecting chords from the diagrams with 2$g$ chords defining the optimal flow on the surface, and color them in green , and the rest of the chord in red.

\subsection{Optimal flows on a complete torus}
There is a single optimal flow on the surface of the torus and its chord diagram has two intersecting chords. By coloring one of the chords in green and the other in red, we will get a colored chord diagram of the flow on a complete torus (body with one handle). Therefore, there is a single optimal flow on the complete torus. Its Pr-diagram was constructed by us earlier in fig. \ref{pic9}.

\subsection{Optimal flows on a body with two handles}

There are 4 topologically inequivalent flows on a closed oriented surface of genus 2 (see, for example, \cite{Kibalko18}). If all the chords intersect at one point (the center of the circle), then we will not be able to choose two non-intersecting chords from them. For the remaining three diagrams, possible colorings are shown in fig.\ref{pic14a}.

\begin{figure}[ht]
\center{\includegraphics[height=2.5cm]{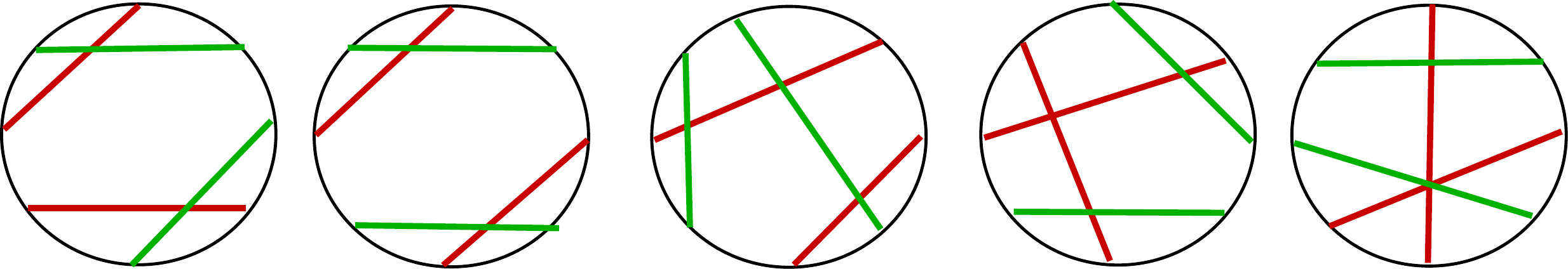}}
\caption{Colored chord diagrams of Morse streams of kind 2 }
\label{pic14a}
\end{figure}

The corresponding Pr flow diagrams are shown in Fig.\ref{pic14b}.

\begin{figure}[ht]
\center{\includegraphics[height=3.5cm]{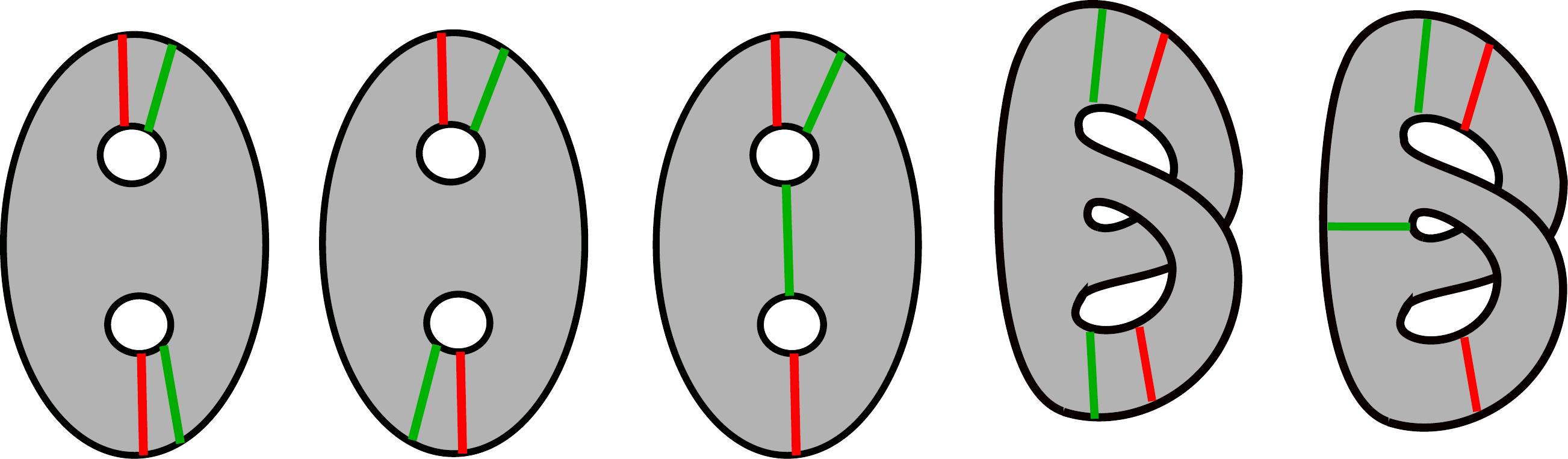}}
\caption{Pr-diagrams of Morse flows on a body with two handles}
\label{pic14b}
\end{figure}

%On an unoriented body with handles of genus 2 exists
%2+3+2+2+3=12
% of optimal flows

Therefore, there are five different structures of optimal flows on a body with two handles.

\subsection{Optimal flows on a body with three handles}

Fig.\ref{pic14c} shows the 82 chord diagrams that define the optimal flows on an oriented surface of genus 3, and also indicates the number of ways in which three non-intersecting chords can be chosen.
Therefore, there are 177 different structures of optimal flows on a body with three handles.

\newpage

\begin{figure}[ht]

%\center{\ \ \ \ \includegraphics[height=9.0cm]{m-hb3a.eps}}
%\caption{Optimal Morse flow on a body with handles of genus 2 }
%\label{pic14b}
%\end{figure}
%
%\begin{figure}[ht]
\center{\includegraphics[height=17.0cm]{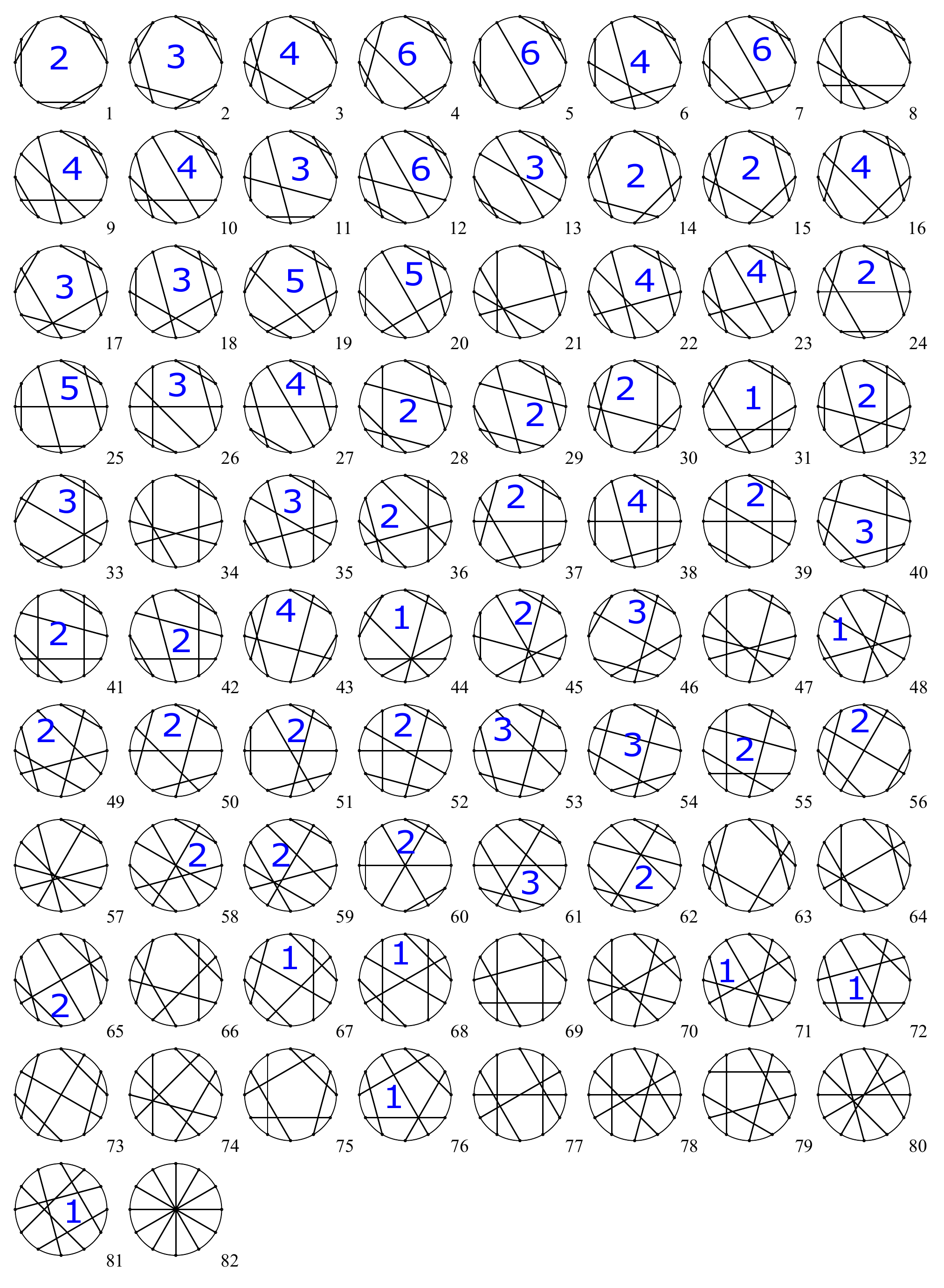}}
\caption{Optimal Morse flows on a body with handles of genus 3}
\label{pic14c}
\end{figure}

\newpage

\subsection{Water flows on a river with islands}

On the river, the flow has zero speed at points on the bank and bottom, but if we observe a body in water of a certain diameter, then we will be interested only in points that do not fall into some small area of the bottom. Therefore, we will discard this circle and smooth the result on the boundary. We project the velocity vector on the resulting boundary surface orthogonally onto the boundary and smooth the vector field around the boundary. Note that the regular trajectories of the border along the bottom approach the shore, and on the surface of the river, they go to the center. If there are $n$ islands on the river, then the flow will be in a body with $n$ handles.

An optimal flow on a body with $n$ handles will be called a river flow if it has the following properties: 1) the boundary surface can be cut into two homeomorphic surfaces, which we will call the upper and lower surfaces, so that all fixed points lie on the intersection of these two surfaces .

\begin{theorem}
A colored chord diagram is a river flow diagram if 1) it has $n$ chords of each color, 2) chords of the same color do not cross each other, 3) on the red chords, you can select one of the ends (upper ends) and arrange the chords as follows, that these ends be consecutive points of the chord diagram (there are no other chord ends, neither red nor green, between them).
\end{theorem}
\begin{proof}
The upper ends correspond to separatrixes that go from the source to the islands along the (upper) surface of the river. Other separatrices coming from the source go along the bottom (lower surface), so their corresponding points on the chord diagram are separated from the highlighted ends. The ends that are not upper are called lower. Then all the chords are divided into pairs: one pair includes a red chord and a green one, which crosses it at the point closest to the upper end.
%Is the split condition not automatically met?
\end{proof}
If we consider water flows on a river with two islands, then among the five diagrams in fig.\ref{pic14b}, only the first and third diagrams are possible for them (islands next to each other and islands one above the other, respectively).

For flows with three islands, among the 82 diagrams in Fig.\ref{pic14c}, only chord diagrams 5, 7, 11, 12, 19, 20, 38, 42 are possible.

\section*{Conclusions}
The complete topological invariant of a Morse flow on oriented 3-manifolds with a boundary, the Pr-flow diagram, which we have constructed, generalizes Heegaard diagrams for closed 3-manifolds.  Flow structures, which were founded with its help on the 3-disc and bodies with handles, prove its effectiveness. It would also be interesting to investigate the following problems:

1) what be the diagrams of optimal flows on other three-dimensional manifolds (for example, a complement to the circle of a node in a three-dimensional sphere; manifolds whose boundary is a sphere),

2) construct a topological invariant of flows on unoriented 3-manifolds,

3) generalize these results to Morse-Smale flows with closed orbits,

4) when there is a continuation of the Morse flow from the surface to the interior of the 3-manifold without internal fixed points,

5) investigate flows on manifolds with angles,

6) investigate flows in which, together with fixed points on the boundary, there are internal fixed points.

%1. Коли дві діаграми Морса  визначають гомеоморфізм тривимірних многовидів?

%2. Чи існує продовження потоку Морса з межі на внутрішність 3—многовида без внутрішніх особливих точок?

%3. Оптимальні потоки Морса на  доповненнях до вузла.

%4. Діаграми потоків Морса-Смейла  з замкненими орбітами.

%5. Коли потік Морса є оптимальним?

\end{document}